\documentclass[12pt]{amsart}
\usepackage{amsmath}
\usepackage{hyperref}
\usepackage{amssymb}
\usepackage{latexsym}
\usepackage{amscd}
\usepackage{graphicx} 
\usepackage{amsthm}
\usepackage{mathrsfs}
\usepackage{xypic}
\usepackage{bm}

\newdimen\AAdi%
\newbox\AAbo%
\def\AAk#1#2{\s_etbox\AAbo=\hbox{#2}\AAdi=\wd\AAbo\kern#1\AAdi{}}%
\def\AAr#1#2#3{\s_etbox\AAbo=\hbox{#2}\AAdi=\ht\AAbo\raise#1\AAdi\hbox{#3}}%
\font\tenmsb=msbm10 at 12pt \font\sevenmsb=msbm7 at 8pt
\font\fivemsb=msbm5 at 6pt
\newfam\msbfam
\textfont\msbfam=\tenmsb \scriptfont\msbfam=\sevenmsb
\scriptscriptfont\msbfam=\fivemsb
\def\Bbb#1{{\tenmsb\fam\msbfam#1}}
\newtheorem{thm}{Theorem}[section]
\newtheorem{lem}{Lemma}[section]

\newtheorem{pro}{Proposition}[section]

\newcommand{\ba}{\begin{array}}
\newcommand{\ea}{\end{array}}

\newcommand{\Section}[2]{\setcounter{equation}{0}
\allowdisplaybreaks
\section[#1]{#2}}

\def\n{\nabla}
\def\bn{\overline\nabla}

\def\ir#1{\mathbb R^{#1}}

\def\f#1#2{\frac{#1}{#2}}

\def\dt#1{\frac {d\,#1}{d\,t}}
\def\mc#1{\mathcal{#1}}

\def\a{\alpha}
\def\be{\beta}

\def\r{\rho}

\def\p#1{\partial #1}

\def\de{\delta}
\def\De{\Delta}
\def\e{\eta}
\def\ep{\varepsilon}
\def\G{\Gamma}
\def\g{\gamma}
\def\k{\kappa}
\def\la{\lambda}

\def\Om{\Omega}

\def\R{\Bbb{R}}

\begin{document}

\title
{The Rigidity Theorems of Self Shrinkers}
\author{Qi Ding}
\author{Y.L. Xin}
\address{Institute of Mathematics, Fudan University,
Shanghai 200433, China} \email{09110180013@fudan.edu.cn}
\email{ylxin@fudan.edu.cn}
\thanks{The research was partially supported by
NSFC}

\begin{abstract}
By using certain idea developed in minimal submanifold theory we
study rigidity problem for self-shrinkers in the present paper. We
prove rigidity results for squared norm of the second fundamental
form of self-shrinkers, either under point-wise conditions or
under integral conditions.
\end{abstract}

\maketitle

\Section{Introduction}{Introduction}

\medskip

The subject of self-shrinkers are closely related with the theory
of minimal submanifolds, as shown by previous works \cite{CM1}
and \cite{CM2}.

There are intrinsic rigidity and extrinsic rigidity  for minimal
submanifolds in the unit sphere.  The intrinsic rigidity implies gap
property of the scalar curvature, so is the squared norm of the
second fundamental form by the Gauss equations. The extrinsic
rigidity describes the gap phenomenon for the image of the Gauss
maps. Both properties of minimal submanifolds were initialed  by J.
Simons in his fundamental paper \cite{Si}. Since then, the extensive
works appeared to contribute to this interesting problem.

Besides the interest in the own right, the rigidity problem in the
sphere is also related to the Benstein problem for minimal
submanifolds in the Euclidean space \cite{X2}.

We now  study the rigidity problem for self-shrinkers. Now, there is
no intrinsic rigidity. However, there also exist gap phenomena for
the squared norm of the second fundamental form  and the image of
the Gauss maps. In the present paper we only pay attention to the
gap phenomenon  for the squared norm of the second fundamental form.
As for the gap phenomenon for the image under the Gauss maps we will
write another paper to contribute to the problem.

The first gap of the squared norm of the second fundamental form for
self-shrinkers was obtained by Cao-Li \cite{CaL} (which generalized
codimension $1$ case in \cite{L-S}).

Chern-doCarmo-Kobayashi in \cite{CDK} confirmed that the Simons
first gap in \cite{Si} is sharp and raised to study the subsequent
gaps.  Peng-Terng in \cite{PT1} and \cite{PT2} studied the second
gap of squared norm of the second fundamental form for compact
minimal hypersurfaces in a unit sphere. They obtained pinching
results for minimal hypersurfaces of constant scalar curvature in
any dimension and that without the constant scalar curvature
assumption in lower dimensions. After that, there are many works on
this problem. Recently, we confirm the second gap in any dimension
without constant scalar curvature assumption \cite{DX1}.

In the present paper we employ  the similar idea in our previous
work \cite{DX1} to study the second gap for self-shrinkers. The
results will be given in the Theorem \ref{sec}. We also study the
self-shrinker surfaces in $\ir{3}$ with constant squared norm of
the second fundamental form. They can be classified, as shown in
the Theorem 4.2.

By using Sobolev's inequality, Ni \cite{N} proved  gap results for
minimal hypersurfaces under the integral conditions on the squared
norm of the second fundamental form. For self-shrinkers , there is
also Sobolev's inequality, which can be used to obtain gap results
for self-shrinkers,  in a  manner analog to  that in \cite{N},  as
shown in Theorem \ref{int}.  But, our direct integral estimates
apply to arbitrary codimension not only for self-shrinkers, but also
for minimal submanifolds with corresponding modifications.

The organization  of the present article is as follows: In next
section, we fix the notations and derive basic formulas in a manner
as in \cite{X2}, which will be used in the later sections. In \S 3,
we prove  rigidity results in higher codimension. In the final
section, we give rigidity results in codimension $1$.

\Section{Preliminaries}{Preliminaries}

\medskip

Let $M$ be an $n$-dimensional Riemannian manifold, and $X: M
\rightarrow\ir{m+n}$ be an isometric immersion. Let $\n$ and
$\bn$ be Levi-Civita connections on $M$ and $\R^{m+n}$,
respectively. The second fundamental form $B$ is defined by
$B(V,W)=(\bn_VW)^N=\bn_VW-\n_VW$ for any vector fields $V,W$
along the submanifold $M$, where $(\cdots)^N$ is the projection
onto the normal bundle $NM$. Similarly, $(\cdots)^T$ stands for
the tangential projection.  Taking the trace of $B$ gives the
mean curvature vector $H$ of $M$ in $\ir{m+n}$,  a cross-section
of the normal bundle. In what follows we use $\n$ for natural
connections on various bundles for notational simplicity if there
is no ambiguity from the context. For $\nu\in\G(NM)$ the shape
operator $A^\nu: TM\to TM$, defined by $A^\nu(V)=-(\bn_V\nu)^T$,
satisfies $\left<B_{V W}, \nu\right>=\left<A^\nu(V), W\right>.$

The second fundamental form, curvature tensors of the
submanifold, curvature tensor of the normal bundle and that of
the ambient manifold satisfy the Gauss equations, the Codazzi
equations and the Ricci equations.

We now consider the mean curvature flow for a submanifold  $M$ in
$\ir{m+n}.$ Namely, consider a one-parameter family $X_t=X(\cdot,
t)$ of immersions $X_t:M\to \ir{m+n}$ with corresponding images
$M_t=X_t(M)$ such that
\begin{equation*}\left\{\begin{split}
\dt{}X(x, t)&=H(x, t),\qquad x\in M\\
X(x, 0)&=X(x)
\end{split}\right.
\end{equation*}
is satisfied, where $H(x, t)$ is the mean curvature vector of
$M_t$ at $X(x, t)$ in $\ir{m+n}.$

An important class of solutions to the above mean curvature flow
equations are self-similar shrinkers, whose profiles,
self-shrinkers, satisfy a system of quasi-linear elliptic PDE of the
second order
\begin{equation}\label{SS}
H = -\frac{X^N}{2}.
\end{equation}

Let $\De$, $\mathrm{div}$ and $d\mu$ be Laplacian, divergence and
volume element on $M$, respectively. Colding and Minicozzi in
\cite{CM1} introduced a linear operator
\begin{equation*}
\mc{L}=\De-\frac{1}{2}\langle
X,\n(\cdot)\rangle=e^{\f{|X|^2}4}\mathrm{div}(e^{-\f{|X|^2}4}\n(\cdot))
\end{equation*}
on self-shrinkers. They showed that $\mc{L}$ is self-adjoint
respect to the measure $e^{-\f{|X|^2}4}d\mu.$  In the present
paper we carry out integrations with respect to this measure. We
denote $\rho=e^{-\f{|X|^2}4}$ and the volume form
 $d\mu$ might be omitted in the integrations for notational simplicity.

In this section we derive several basic formulas for
self-shrinkers. Some of them have been known in the literature.
For convenience, we derive them here in our notations.

For minimal submanifolds in an arbitrary ambient Riemannian
manifold J.Simons \cite{Si} derived the Laplacian of the squared
norm of the second fundamental form. For arbitrary submanifolds
in Euclidean space Simons type formula was also  derived (see
\cite{Sm}, \cite{X1}, for example).

Choose a local orthonormal frame field $\{e_i,e_\a\}$ along $M$
with dual frame field $\{\omega_i,\omega_\a\}$, such that $e_i$
are tangent vectors of $M$ and $e_{\a}$ are normal to $M$. The
induced Riemannian metric of $M$ is given by $ds_M^2
=\sum_i\omega_i^2$ and the induced structure equations of $M$ are
\begin{equation*}\begin{split}
   & d\omega_i = \omega_{ij}\wedge\omega_j,\qquad
                 \omega_{ij}+\omega_{ji} = 0,\cr
   &d\omega_{ij}= \omega_{ik}\wedge\omega_{kj}+\omega_{i\a}\wedge\omega_{\a j},\cr
   &\Omega_{ij} = d\omega_{ij}-\omega_{ik}\wedge\omega_{kj}
                = -\frac 12 R_{ijkl}\omega_k\wedge\omega_l.
\end{split}\end{equation*}
By Cartan's lemma we have
$$\omega_{\a i} = h_{\a ij}\omega_j.$$
Here and in the sequel we agree with the following range of
indices
$$1\le i, j, k, \cdots \le n,\quad n+1\le \a, \be, \g, \cdots \le n+m.$$

Set
$$B_{ij}=B_{e_ie_j}=(\overline{\n}_{e_i}e_j)^N=h_{\a ij}e_\a,\quad S_{\a\be}=h_{\a ij}h_{\be ij}.$$
Then, $$|B|^2=\sum_{\a}S_{\a\a}.$$

From Proposition 2.2 in \cite{X1} we have
\begin{equation}\label{2.2}\begin{split}
\De|B|^2=2\,|\n B|^2+2\,\left<\n_i\n_jH,B_{ij}\right>
       &+\,2\left<B_{ij},H\right>\left<B_{ik},B_{jk}\right>\\
       &\quad -\,2\,\sum_{\a\ne\be}|[A^{e_\a}, A^{e_\be}]|^2-2\,\sum_{\a,\be}S_{\a\be}^2.
\end{split}\end{equation}

We suppose that the local orthonormal frame field
$\{e_i\}_{i=1}^{n}$ is normal at a considered point $p\in M$. From
the self-shrinker equations (\ref{SS}) we obtain
\begin{equation}\label{DH}
\n_jH=\f{1}{2}\left<X, e_k\right>B_{jk},
\end{equation}
and
\begin{equation}\label{D2H}
\n_i\n_j H=\f{1}{2}B_{ij}-\langle H, B_{ik}\rangle
B_{jk}+\f{1}{2}\langle X,e_k\rangle\n_i B_{jk}.
\end{equation}
Combining (\ref{2.2}) and (\ref{D2H})(and using the Codazzi  equation), we have
\begin{equation}\aligned\label{LB}
\mc{L}|B|^2=2|\n B|^2+|B|^2-\,2\,\sum_{\a\ne\be}|[A^{e_\a},
A^{e_\be}]|^2-2\,\sum_{\a,\be}S_{\a\be}^2.
\endaligned
\end{equation}
This is the self-shrinker version of the well-known Simons'
identity. In particular, when the codimension $m=1$, the above
Simons' type identity reduces to the following one:

\begin{equation}\label{LB1}
\mc{L}|B|^2=2|\n B|^2 + 2|B|^2\left(\f{1}{2}- |B|^2\right).
\end{equation}

In general, we know from \cite{Si}
$$\sum_{\a\ne\be}|[A^{e_\a}, A^{e_\be}]|^2+\sum_{\a, \be}S_{\a\be}^2\le \left(2-\frac 1n\right)|B|^4.$$
When the codimension $m\ge 2$ the above estimate was refined
\cite{LL}\cite{ChenX}
$$\sum_{\a\ne\be}|[A^{e_\a}, A^{e_\be}]|^2+\sum_{\a, \be}S_{\a\be}^2\le\frac 32 |B|^4.$$
Combining (\ref{LB}) and the above inequality, we have
\begin{equation}\aligned\label{LB2}
\mathcal{L}|B|^2\ge 2|\n|B||^2+|B|^2-3|B|^4,
\endaligned
\end{equation}
here we use rough estimates $|\n B|^2\ge |\n|B||^2$. It can be
refined by so-called Kato's type inequality.

From (\ref{D2H}) (and using the Codazzi  equation) we have
\begin{equation*}\aligned
\De|H|^2&=2\langle H,\n^2 H\rangle+2|\n H|^2\\
&=\left\langle H, H-2\langle H, B_{ik}\rangle
B_{ik}+\langle X,e_k\rangle\n_{e_k} H\right\rangle+2|\n H|^2\\
&=|H|^2-2\sum_{i,j}|\langle H,B_{ij}\rangle|^2+\f{1}{2}\langle X,\n|H|^2\rangle+2|\n H|^2.\\
\endaligned
\end{equation*}
It follows that
\begin{equation}\label{LH}
\mc{L}|H|^2=|H|^2-2\sum_{i,j}|\langle H,B_{ij}\rangle|^2+2|\n
H|^2.
\end{equation}

\bigskip

\Section{Rigidity results in High Codimension}{Rigidity results
in High Codimension}

\medskip

First of all we use formula (\ref{LH}) to obtain a rigidity
result for the squared norm of the second fundamental form which
was already known \cite{CaL}.
\begin{pro}\label{1gap}
Let $M^n$ be a complete properly immersed self-shrinker in
$\R^{n+m}$ with $|B|^2\leq\frac{1}{2}$, then either $|B|\equiv0$,
and $M$ is a n-plane or $|B|^2\equiv\frac{1}{2}$, and $M$ is a
product $S^k(\sqrt{2k})\times\R^{n-k}$ for $1\leq k\leq n$.
\end{pro}
\begin{proof}
Let $\eta$ be a smooth function with compact support in $M$, then
by (\ref{LH}), we have
\begin{equation}\aligned\label{2.7}
&\int_M(\frac{1}{2}|H|^2-\sum_{i,j}|\langle H,B_{ij}\rangle|^2+|\nabla H|^2)\eta^2\rho\\
=&\frac{1}{2}\int_M(\mc{L}|H|^2)\eta^2\rho
=\frac{1}{2}\int_M\mathrm{div}(\rho\nabla|H|^2)\eta^2\\
=&-\int_M\eta\rho\nabla|H|^2\cdot\nabla\eta
\leq\frac{1}{2}\int_M|\nabla
H|^2\eta^2\rho+2\int_M|H|^2|\nabla\eta|^2\rho.
\endaligned
\end{equation}
Since
\begin{equation}\aligned\label{HB}
\sum_{i,j}|\langle H,B_{ij}\rangle|^2\le |H|^2|B|^2,
\endaligned
\end{equation}
we then have
\begin{equation}\aligned\label{pinch}
\int_M|H|^2(\frac{1}{2}-|B|^2)\eta^2\rho+\frac{1}{2}\int_M|\nabla
H|^2\eta^2\rho\leq 2\int_M|H|^2|\nabla\eta|^2\rho.
\endaligned
\end{equation}
If $M$ is compact, we choose $\eta\equiv 1$, otherwise, let
$\eta(X)=\eta_r(X)=\phi(\f{|X|}r)$ for any $r>0$, where $\phi$ is a
nonnegative function on $[0,+\infty)$ satisfying
\begin{eqnarray}\label{eta}
   \phi(x)= \left\{\begin{array}{ccc}
     1     & \quad\ \ \ {\rm{if}} \ \ \  x\in [0,1) \\ [3mm]
     0  & \quad\quad\  {\rm{if}} \ \ \  x\in [2,+\infty),
     \end{array}\right.
\end{eqnarray}
and $|\phi'|\le C$ for some absolute constant.
Noting the Euclidean volume growth of $M$ by our previous result in
\cite{DX2} and $|H|\le\f12|X|$, the right hand side of \eqref{pinch}
approaches to zero as $r\rightarrow+\infty$. This implies that
$H^2(\frac{1}{2}-|B|^2)\equiv0$ and $|\nabla H|\equiv0$. Since
$\nabla|H|^2=2\langle H,\nabla H\rangle$, then $|H|$ is a constant.
If $|H|\equiv0$, then $M$ is a $n-$plane. Otherwise $|H|>0$ and
$|B|^2=\frac{1}{2}$. Moreover, \eqref{HB} takes equality, which
implies $B_{ij}=\langle B_{ij},\nu\rangle\nu$ for any $i,j$,
$\nu=\f{H}{|H|}$. By Theorem 1 of Yau in \cite{Y}, $M$ lies  some
$n+1$-dimensional linear subspace $\R^{n+1}$. From (\ref{LB1}),
$|\nabla B|\equiv0$ which implies that the eigenvalues of $B$ are
constants on $M$. In Theorem 4 of \cite{L}, Lawson showed that every
smooth hypersurface with $\nabla B= 0$ splits isometrically as a
product of a sphere and a linear space(i.e. $S^k\times \R^{n-k}$).
Furthemore, by the  self-shrinker equation \eqref{SS}, the
$k$-dimensional sphere should has the radius  $\sqrt{2k}$ and
centered at the origin.
\end{proof}


There is a Sobolev inequality (see  \cite{MS})
 as follows
\begin{equation}\aligned\label{2.8}
\k^{-1}\left(\int_Mg^\frac{2n}{n-2}\mathrm{d}\mu\right)^\frac{n-2}{n}\leq\int_M|\nabla
g|^2\mathrm{d}\mu+\frac{1}{2}\int_M|H|^2g^2\mathrm{d}\mu,\qquad
\forall g\in C_c^\infty(M),
\endaligned
\end{equation}
where  $\kappa>0$ is a constant.  Besides using (\ref{2.8}), the
Simons type inequality in self-shrinker version \eqref{LB2} would
be used in the following result.
\begin{thm}\label{int}
Let $M^n$ be a complete immersed self-shrinker in $\R^{n+m}.$ If
$M$ satisfies an integral condition
$\left(\int_M|B|^n\mathrm{d}\mu\right)^{1/n}<\sqrt{\frac{4}{3n\k}}$,
then $|B|\equiv0$ and $M$ is a linear subspace.
\end{thm}
\begin{proof}
Let $\eta$ be a smooth function with compact support in $M$.
Multiplying $\eta^2|B|^{n-2}$ on both sides of \eqref{LB2} and
integrating by parts yield
\begin{equation}\aligned
0\ge&2\int_M|\n|B||^2|B|^{n-2}\e^2\rho+\int_M|B|^n\e^2\r-3\int_M|B|^{n+2}\e^2\rho-\int_M\e^2|B|^{n-2}\mathcal{L}|B|^2\\
=&2\int_M|\n|B||^2|B|^{n-2}\e^2\rho+\int_M|B|^n\e^2\rho-3\int_M|B|^{n+2}\e^2\rho\\
&+2\int_M|B|\r\n|B|\cdot\n(|B|^{n-2}\e^2)\\
=&2(n-1)\int_M|\n|B||^2|B|^{n-2}\e^2\rho+\int_M|B|^n\e^2\rho-3\int_M|B|^{n+2}\e^2\rho\\
&+4\int_M(\n|B|\cdot\n\e)|B|^{n-1}\eta\rho.\\
\endaligned
\end{equation}
By Cauchy inequality, for any $\ep>0$, we have
\begin{equation}\aligned\label{2.10}
3\int_M|B|^{n+2}\e^2\rho-\int_M|B|^n\e^2\rho&+\frac{2}{\ep}\int_M|B|^n|\n\e|^2\r\\
\ge&2(n-1-\ep)\int_M|\n|B||^2|B|^{n-2}\e^2\rho.\\
\endaligned
\end{equation}
Let $f=|B|^{n/2}\r^{1/2}\e.$ Integrating by parts, then we
have
\begin{equation}\aligned\label{2.11}
\int_M|\n f|^2=&\int_M|\n(|B|^{n/2}\e)|^2\r+\f12\int_M\n (|B|^n\e^2)\cdot\n\r+\int_M|B|^n\e^2|\n\r^{1/2}|^2\\
=&\int_M|\n(|B|^{n/2}\e)|^2\r-\f12\int_M|B|^n\e^2\De\r+\f1{16}\int_M|B|^n\e^2|X^T|^2\r.
\endaligned
\end{equation}
By \eqref{SS}, we have $\De|X|^2=2n-|X^N|^2$(see \cite{CM1} or
\cite{DX2}), then
$$\De\r=-\f{\r}4\De|X|^2+\f{\r}{16}|\n|X|^2|^2=-\f{\r}4(2n-|X^N|^2)+\f{\r}4|X^T|^2=-\f
n2\r+\f\r4|X|^2.$$ From \eqref{2.11}, we get(see also \cite{E})
\begin{equation}\aligned\label{2.12}
\int_M|\n f|^2=&\int_M|\n(|B|^{n/2}\e)|^2\r-\f18\int_M|B|^n\e^2|X^N|^2\r\\
&+\f n4\int_M|B|^n\e^2\r-\f1{16}\int_M|B|^n\e^2|X^T|^2\r.
\endaligned
\end{equation}
Combining \eqref{SS}, Sobolev inequality (\ref{2.8}) and
\eqref{2.12}, we have
\begin{equation}\aligned\label{2.13}
&\k^{-1}\left(\int_M|f|^{\f {2n}{n-2}}\right)^{\f{n-2}n}\le\int_M|\n f|^2+\f18\int_M|B|^n\e^2|X^N|^2\r\\
\le&\int_M|\n(|B|^{n/2}\e)|^2\r+\f n4\int_M|B|^n\e^2\r\\
=&\int_M(\f{n^2}4|\n|B||^2|B|^{n-2}\e^2+n|B|^{n-1}\e\n|B|\cdot\n\e+|B|^n|\n\e|^2)\r+\f n4\int_M|B|^n\e^2\r.\\
\endaligned
\end{equation}
Combining Cauchy inequality, \eqref{2.10} and \eqref{2.13}, for
any $\de>0$, we have
\begin{equation}\aligned\label{2.14}
&\k^{-1}\left(\int_M|f|^{\f {2n}{n-2}}\right)^{\f{n-2}n}\\
\le&(1+\de)\f{n^2}4\int_M|\n|B||^2|B|^{n-2}\e^2\r+(1+\f1{\de})\int_M|B|^n|\n\e|^2\r+\f n4\int_M|B|^n\e^2\r\\
\le&\f{(1+\de)n^2}{8(n-1-\ep)}\left(3\int_M|B|^{n+2}\e^2\r-\int_M|B|^n\e^2\r
+\frac{2}{\ep}\int_M|B|^n|\n\e|^2\r\right)\\
&+(1+\f1{\de})\int_M|B|^n|\n\e|^2\r+\f n4\int_M|B|^n\e^2\r.\\
\endaligned
\end{equation}
Let $\de=2\f{n-1+\ep}n-1>0$ in \eqref{2.14} for some $\ep>0$ to be defined later, then
\begin{equation}\aligned\label{2.15}
&\k^{-1}\left(\int_M|f|^{\f {2n}{n-2}}\right)^{\f{n-2}n}\\
\le&\f {3n}4\cdot\f{n-1+\ep}{n-1-\ep}\int_M|B|^{n+2}\e^2\r+(\f n{2\ep}\cdot\f{n-1+\ep}{n-1-\ep}+1+
\f1{\de})\int_M|B|^n|\n\e|^2\r\\
\le&\f {3n}4\cdot\f{n-1+\ep}{n-1-\ep}\left(\int_M|B|^{2\cdot\f n2}\right)^{\f2 n}\left(\int_M(|B|^{n}\e^2\r)^{\f n{n-2}}\right)^{\f {n-2}n}\\
&+(\f n{2\ep}\cdot\f{n-1+\ep}{n-1-\ep}+1+\f1{\de})\int_M|B|^n|\n\e|^2\r.\\
\endaligned
\end{equation}
Since $(\int_M|B|^n\mathrm{d}\mu)^{1/n}<\sqrt{\frac{4}{3n\k}}$,
then from \eqref{2.15} there is $0<\ep_0<1$ such that
\begin{equation}\aligned\nonumber
\k^{-1}\left(\int_M|f|^{\f {2n}{n-2}}\right)^{\f{n-2}n}\le&\f {3n}4\cdot\f{n-1+\ep}{n-1-\ep}\cdot\frac{4(1-\ep_0)}{3n\k}\left(\int_M|f|^{\f {2n}{n-2}}\right)^{\f {n-2}n}\\
&+C(\ep)\int_M|B|^n|\n\e|^2\r,\\
\endaligned
\end{equation}
namely,
\begin{equation}\aligned
\f{(n-1+\ep)\ep_0-2\ep}{(n-1-\ep)\k}\left(\int_M|f|^{\f {2n}{n-2}}\right)^{\f{n-2}n}\le C(\ep)\int_M|B|^n|\n\e|^2\r.\\
\endaligned
\end{equation}
Let $\ep=\f {\ep_0}2$, since $\int_M|B|^n\mathrm{d}\mu$ is
bounded, then we choose $\e$ as Proposition \ref{1gap} which
implies $|B|\equiv0$.
\end{proof}

\begin{rmk}
For codimension $m=1$ case, we can use (\ref{LB1}) instead of (\ref{LB2}) and the pinching constant would be better.

\end{rmk}

\bigskip

\Section{Rigidity results for Codimension 1}{Rigidity results for
Codimension 1}

\medskip

Now, we deal with the codimension $m=1$ case. We choose a local
orthonormal frame field $\{e_1, \cdots, e_n, \nu\}$ in $\R^{n+1}$
along the hypersurce  $M$ with $\{e_i\}_{i=1}^{n}$  tangent to
$M$ and $\nu$  normal to $M$. Set the second fundamental form
$B_{e_ie_j}=h_{ij}\nu$.

Define the covariant derivatives $Dh$ of $h$ (with component
$h_{ijk}$) by
$$\sum_kh_{ijk}\omega_k=dh_{ij}-\sum_kh_{ik}\omega_{jk}-\sum_kh_{kj}\omega_{ik},$$
and similarly we can define the covariant derivatives $h_{ijkl}$ and
$h_{ijkls}$. We have the Ricci identity
\begin{equation}\aligned\label{Ric id}
h_{ijkl}-h_{ijlk}&=\sum_{s=1}^nh_{is}R_{sjkl}+\sum_{s=1}^nh_{sj}R_{sikl},\\
h_{ijkls}-h_{ijksl}&=\sum_{r=1}^nh_{rjk}R_{rils}+\sum_{r=1}^nh_{irk}R_{rjls}+\sum_{r=1}^nh_{ijr}R_{rkls}.
\endaligned
\end{equation}

We need the following higher order Simons' type formula for
further estimates.

\begin{thm}
Let $M^n$ be an immersed self-shrinker in $\R^{n+1}.$ Then, we
have
\begin{equation}\aligned\label{2 diff}
\sum_{i,j,k,l}h_{ijkl}^2-\frac{1}{2}\mathcal{L}|\nabla
B|^2=(|B|^2-1)|\nabla
B|^2+3\Xi+\frac{3}{2}|\nabla|B|^2|^2,
\endaligned
\end{equation}
where $\Xi=\sum_{i,j,k,l,m}h_{ijk}h_{ijl}h_{km}h_{ml}-2\sum_{i,j,k,l,m}h_{ijk}h_{klm}h_{im}h_{jl}.$
\end{thm}
\begin{rmk}
At each point $p\in M$, $h_{ij}$ can be diagonalized $h_{ij}=\la_i\de_{ij}$. Then,
$$\Xi=\sum_{i,j,k}h_{ijk}^2(\lambda^2_k-2\lambda_i\lambda_j).$$
\end{rmk}
\begin{proof}
We choose a local orthonormal frame field $\{e_i\}_{i=1}^{n}$ and
normal at a considered point $p\in M$, i.e.,
$\left.\n_{e_i}e_j\right|_p=0$ for any $1\le i,j\le n$. By Ricci
identity \eqref{Ric id}, we obtain
\begin{equation}\aligned\label{3.2}
\Delta
h_{ijk}=&h_{ijkll}=(h_{ijlk}+h_{ir}R_{rjkl}+h_{rj}R_{rikl})_l\\
=&h_{ijllk}+h_{rjl}R_{rikl}+h_{irl}R_{rjkl}+h_{ijr}R_{rlkl}+(h_{ir}R_{rjkl}+h_{rj}R_{rikl})_l\\
=&(h_{ljli}+h_{lr}R_{rjil}+h_{rj}R_{rlil})_k+h_{rjl}R_{rikl}+h_{irl}R_{rjkl}+h_{ijr}R_{rlkl}\\
&+h_{irl}R_{rjkl}+h_{rjl}R_{rikl}+h_{ir}(R_{rjkl})_l+h_{rj}(R_{rikl})_l\\
=&H_{jik}+h_{rkl}R_{rjil}+h_{rjk}R_{rlil}+2h_{rjl}R_{rikl}+2h_{ril}R_{rjkl}+h_{rij}R_{rlkl}\\
&+h_{ir}(R_{rjkl})_l+h_{rj}(R_{rikl})_l+h_{lr}(R_{rjil})_k+h_{rj}(R_{rlil})_k.
\endaligned
\end{equation}
By (\ref{D2H}) (when the codimension is 1),
\begin{equation}\aligned\label{3.3}
2H_{ji}=h_{jli}\langle X,e_l\rangle+h_{ij}-2Hh_{il}h_{jl}.
\endaligned
\end{equation}
Since $-\frac{1}{2}\langle X,\nu\rangle=H=\sum_ih_{ii}$, then
\begin{equation}\aligned\label{3.4}
2H_{jik}=&h_{jlik}\langle X,e_l\rangle+h_{jli}\langle
e_k,e_l\rangle+h_{jli}\langle X,\overline{\nabla}_{e_k}e_l
\rangle+h_{ijk}\\
&-2H_kh_{il}h_{jl}-2H(h_{ikl}h_{jl}+h_{il}h_{jkl})\\
=&h_{jlik}\langle X,e_l\rangle+2h_{ijk}-2H_kh_{il}h_{jl}
-2H(h_{il}h_{jkl}+h_{jl}h_{ikl}+h_{kl}h_{jli}).
\endaligned
\end{equation}
Combining (\ref{3.2}) and (\ref{3.4}), we obtain
\begin{equation}\aligned\label{3.5}
\Delta h_{ijk}=&\frac{1}{2}h_{jlik}\langle
X,e_l\rangle+h_{ijk}-H(h_{il}h_{jlk}+h_{jl}h_{ikl}+h_{kl}h_{jli})-h_{il}h_{jl}H_{k}\\
&+h_{rkl}R_{rjil}+h_{rjk}R_{rlil}+2h_{rjl}R_{rikl}+2h_{ril}R_{rjkl}+h_{rij}R_{rlkl}\\
&+h_{ir}(R_{rjkl})_l+h_{rj}(R_{rikl})_l+h_{lr}(R_{rjil})_k+h_{rj}(R_{rlil})_k.
\endaligned
\end{equation}
A straightforward computation gives
\begin{equation}\aligned\label{3.6}
& h_{ijk}(h_{rkl}R_{rjil}+h_{rjk}R_{rlil}+2h_{rjl}R_{rikl}+2h_{ril}R_{rjkl}+h_{rij}R_{rlkl}\\
&+h_{ir}(R_{rjkl})_l+h_{rj}(R_{rikl})_l+h_{lr}(R_{rjil})_k+h_{rj}(R_{rlil})_k)\\
=& h_{ijk}(6h_{rkl}h_{ri}h_{jl}-6h_{rkl}h_{rl}h_{ij}+3h_{rij}h_{rk}h_{ll}-3h_{rij}h_{rl}h_{kl}\\
&+3h_{ir}h_{rk}h_{jll}-2h_{ir}h_{jk}h_{rll}-h_{ijk}h_{rl}^2).
\endaligned
\end{equation}
From Ricci identity \eqref{Ric id} and (\ref{DH}), we have
\begin{equation}\aligned\label{3.7}
\frac{h_{ijk}}{2}(h_{ijlk}-h_{ijkl})\langle X,e_l\rangle=&\frac{h_{ijk}}{2}(h_{ir}R_{rjlk}+h_{rj}R_{rilk})\langle X,e_l\rangle\\
=&2h_{ijk}h_{ir}h_{jk}H_r-2h_{ijk}h_{ir}h_{rk}H_j.
\endaligned
\end{equation}
By \eqref{3.5}-\eqref{3.7}, we have
\begin{equation}\aligned\label{3.8}
&\frac{1}{2}(\Delta-\frac{1}{2}\langle
X,\nabla\cdot\rangle)h_{ijk}^2=h_{ijk} (\Delta
h_{ijk}-\frac{1}{2}h_{ijkl}\langle X,e_l\rangle)+h_{ijkl}^2\\
=&\frac{h_{ijk}}{2}(h_{ijlk}-h_{ijkl})\langle
X,e_l\rangle+h_{ijk}^2+h_{ijkl}^2-Hh_{ijk}(h_{il}h_{jlk}+h_{jl}h_{ikl}+h_{kl}h_{jli})\\
&-h_{ijk}h_{il}h_{jl}H_k+h_{ijk}(h_{rkl}R_{rjil}+h_{rjk}R_{rlil}+2h_{rjl}R_{rikl}+2h_{ril}R_{rjkl}\\
&+h_{rij}R_{rlkl}
+h_{ir}(R_{rjkl})_l+h_{rj}(R_{rikl})_l+h_{lr}(R_{rjil})_k+h_{rj}(R_{rlil})_k)\\
=&2h_{ijk}h_{ir}h_{jk}H_r-2h_{ijk}h_{ir}h_{rk}H_j+h_{ijk}^2+h_{ijkl}^2-3Hh_{ijk}h_{ijl}h_{kl}-h_{il}h_{jl}H_kh_{ijk}\\
&+h_{ijk}(6h_{rkl}h_{ri}h_{jl}-6h_{rkl}h_{rl}h_{ij}+3h_{rij}h_{rk}h_{ll}-3h_{rij}h_{rl}h_{kl}+3h_{ir}h_{rk}h_{jll}\\
&-2h_{ir}h_{jk}h_{rll}-h_{ijk}h_{rl}^2)\\
=&(1-|B|^2)h_{ijk}^2+h_{ijkl}^2+h_{ijk}(6h_{iu}h_{jv}h_{uvk}-3h_{iju}h_{uv}h_{kv})-\frac{3}{2}|\nabla|B|^2|^2.\\
\endaligned
\end{equation}
\end{proof}

The quantity $\Xi$ can be estimated as follows.

For $k\neq j$, by Cauchy inequality, we have
$$\la_k^2-2\la_k\la_j\le\la_k^2+\f{\sqrt{17}-1}2\la_k^2+\f{\sqrt{17}+1}2\la_j^2\le\f{\sqrt{17}+1}2|B|^2.$$
Then
\begin{equation}\aligned\label{pt}
3\Xi\leq&\sum_{i,j,k\ distinct}h_{ijk}^2\left(\lambda_i^2+\lambda_j^2+\lambda_k^2
-2(\lambda_i\lambda_j+\lambda_j\lambda_k+\lambda_i\lambda_k)\right)
+3\sum_{j,i\neq j}h_{iij}^2(\lambda_j^2-4\lambda_i\lambda_j)\\
\leq&\sum_{i,j,k\ distinct}h_{ijk}^2\left(2(\lambda_i^2+\lambda_j^2+\lambda_k^2)
-(\lambda_i+\lambda_j+\lambda_k)^2\right)+3\sum_{j,i\neq j}h_{iij}^2\f{\sqrt{17}+1}2|B|^2\\
\leq&2|B|^2\sum_{i,j,k\
distinct}h_{ijk}^2+\f{3(\sqrt{17}+1)}2|B|^2\sum_{j,i\neq
j}h_{iij}^2\le \frac{\sqrt{17}+1}{2}|B|^2|\nabla B|^2.
\endaligned
\end{equation}

\begin{thm}
Let  $M^2$ be a complete proper self-shrinker in $\R^3.$ If the
squared norm  $|B|^2$ of the length of second fundamental form is
a constant, then $|B|^2\equiv0$ or $\frac{1}{2}$.
\end{thm}
\begin{proof}
If mean curvature $H$ is non-positive, then by Huisken's
classification theorem (see \cite{CM1}\cite{H1}\cite{H2}) and
Euclidean volume growth \cite{DX2}, we know Abresch-Langer curve \cite{AL} has not
constant curvature, then $M$ is isometric to $S^k\times\R^{2-k}$ for
$0\le k\le 2$. Hence $|B|^2\equiv\frac{1}{2}$ or 0.

Now, we suppose that the mean curvature $H$ changes sign and
$|B|^2>\f12$. For any fixed point $p$ with mean curvature $H|_p=0$,
we suppose that $\{e_1,e_2\}$ is normal at the point $p$ and
$h_{ij}=\la_i\de_{ij}$ for $i=1,2$, then
\begin{equation}\label{c1}
\la_1+\la_2=0
\end{equation}
at the point $p$.  In this proof, we always carry out derivatives at
$p.$

By
$0=\f12(|B|^2)_k=\sum_{i,j}h_{ij}h_{ijk}=\la_1h_{11k}+\la_2h_{22k}$
and \eqref{c1}, we have
\begin{equation}\label{c2}
h_{111}=h_{122},\ \ h_{112}=h_{222}.
\end{equation}
Then
\begin{equation}\label{c3}
|\n B|^2=\sum_{i,j,k}h_{ijk}^2=4h_{111}^2+4h_{222}^2.
\end{equation}
Since
$$h_{11i}+h_{22i}=-\f{(\left< X,\nu\right>)_i}2=-\f{\left< X,\n_{e_i}\nu\right>}2
=\f{\left< X,e_j\right>}2h_{ij}=\f{\left<
X,e_i\right>}{2}\la_{i},$$ and denote $\left< X,e_i\right>$ by
$x_i$, then by \eqref{c2}, we have
\begin{equation}\label{c4}
h_{111}=\f14x_1\la_1,\ \ h_{222}=\f14x_2\la_2.
\end{equation}
Combining \eqref{c1}, \eqref{c3}, \eqref{c4} and $H=-\f{\left<
X,\nu\right>}2=0$, we get
\begin{equation}\label{c5}
|\n
B|^2=\f14x_1^2\la_1^2+\f14x_2^2\la_2^2=\f14|X|^2\la_1^2=\f18|X|^2|B|^2.
\end{equation}
By (\ref{LB1}), $|\n B|^2=|B|^2(|B|^2-\f12)$, we obtain
\begin{equation}\label{c6}
|X|^2=8(|B|^2-\f12)=16\la_1^2-4.
\end{equation}
From Ricci identity \eqref{Ric id}, we can obtain
\begin{equation}\label{c7}
h_{ijkl}-h_{ijlk}=(\la_i-\la_j)\la_i\la_j(\de_{ik}\de_{jl}-\de_{il}\de_{jk}).
\end{equation}
Especially, $h_{iikl}=h_{iilk}$. Moreover,
\begin{equation}\label{c8}
0=\f12(|B|^2)_{kl}=\sum_{i,j}(h_{ijk}h_{ijl}+h_{ij}h_{ijkl}).
\end{equation}
Combining \eqref{3.3} and \eqref{c8}, we have
\begin{equation}\aligned\label{c8a}
h_{11kl}&=\f14(x_1h_{kl1}+x_2h_{kl2}+h_{kl})-\f1{2\la_1}\sum_{i,j}h_{ijk}h_{ijl},\\
h_{22kl}&=\f14(x_1h_{kl1}+x_2h_{kl2}+h_{kl})-\f1{2\la_2}\sum_{i,j}h_{ijk}h_{ijl}.
\endaligned
\end{equation}
Combining \eqref{c1}-\eqref{c7} and \eqref{c8a}, we have
\begin{equation}\aligned\label{c9}
&h_{1111}=\f{\la_1}4-\f{\la_1}8x_2^2;\qquad h_{1122}=-\f{\la_1}4-\f{\la_1}8x_2^2;\\
&h_{2211}=-\f{\la_2}4-\f{\la_2}8x_1^2;\qquad h_{2222}=\f{\la_2}4-\f{\la_2}8x_1^2;\\
&h_{1112}=h_{1121}=\f18x_1x_2\la_1;\qquad
h_{2212}=h_{2221}=\f18x_1x_2\la_2.
\endaligned
\end{equation}
By \eqref{c6} and \eqref{c9}, we get
\begin{equation}\aligned\label{c10}
|\n^2
B|^2=\sum_{i,j,k,l}h_{ijkl}^2=\f{\la_1^2}{16}(|X|^4+2|X|^2+8)=\la_1^2(16\la_1^4-6\la_1^2+1),
\endaligned
\end{equation}
and
\begin{equation}\aligned\label{c11}
\sum_{i,j,k}h_{ijk}^2(\la_k^2-2\la_i\la_j)=|\n B|^2\la_1^2.
\endaligned
\end{equation}
By formula \eqref{2 diff},
we have
\begin{equation}\aligned\label{c12}
\la_1^2(16\la_1^4-6\la_1^2+1)=2\la_1^2(2\la_1^2-\f12)(2\la_1^2-1)+6\la_1^2(2\la_1^2-\f12)\la_1^2,
\endaligned
\end{equation}
which implies
\begin{equation}\aligned\label{c13}
\la_1^2=\f34.
\endaligned
\end{equation}
By \eqref{c6}, we get
\begin{equation}\aligned\label{c14}
|X|^2=8|B|^2-4=8.
\endaligned
\end{equation}
By the formula (\ref{D2H}), $\mathcal{L}H+(|B|^2-\f12)H=H$(see also \cite{CM1}),
we have
\begin{equation}\aligned\label{c.15}
\mathcal{L}H+H=0.
\endaligned
\end{equation}
Let the set $E=\{p\in M;\ H(p)=0\}.$ Since $H$ changes sign and \eqref{c14}, $E\neq\emptyset$, $\p
E=\emptyset$, $E\subset\p D_{2\sqrt{2}}$ and $H(p)\neq0$ for any
$p\in D_{2\sqrt{2}}$. Then there is a constant $c_1\ge1$ and an
eigenfunction $u_1>0$ in some connect component $\Om_1$ of
$D_{2\sqrt{2}}$ such that
\begin{eqnarray*}
    \left\{\begin{array}{ccc}
     \mathcal{L}u_1+c_1u_1=0     & \quad\ \ \ \mathrm{in}\ \Om_1 \\ [3mm]
     u_1|_{\p \Om_1}=0.  & \quad\quad\
     \end{array}\right.
\end{eqnarray*}


Let $g=4-|X|^2$, then
$\mathcal{L}g=-\mathcal{L}|X|^2=|X|^2-4=-g$. There is a constant
$c_2\in(0,1]$ and an eigenfunction $u_2>0$ in some connect
component $\Om_2$ of $D_2$ with $\Om_2\subset\Om_1$ in $D_2$ such that
\begin{eqnarray*}
    \left\{\begin{array}{ccc}
     \mathcal{L}u_2+c_2u_2=0     & \quad\ \ \ \mathrm{in}\ \Om_2 \\ [3mm]
     u_2|_{\p \Om_2}=0.  & \quad\quad\
     \end{array}\right.
\end{eqnarray*}
By the Rayleigh quotient characterization of the first eigenvalue,
we know the first eigenvalue is decreasing in domains. The above
argument contradicts with this fact. Therefore, the case $H$
changing sign and $|B|^2>\f{1}{2}$ is impossible.

\end{proof}

We define the traceless part of second fundamental form by
$\dot{B}=B-\f1n gH$, where $g$ is the metric of $M$. Then we have
$$|\dot{B}|^2=|B|^2-\f{|H|^2}n,\qquad \mathrm{and}\qquad |\n\dot{B}|^2=|\n B|^2-\f{|\n H|^2}n.$$
In face, by \cite{Hu}, the tensor $\n B$ could be decomposed into
orthogonal components $\n_iB_{jk}=E_{ijk}+F_{ijk}$ where
$$E_{ijk}=\f1{n+2}(g_{ij}\n_kH+g_{ik}\n_jH+g_{jk}\n_iH)\qquad\text{and}\qquad |E|^2=\f{3}{n+2}|\n H|^2.$$ Then
\begin{equation}\aligned
|\n \dot{B}|^2=&|\n(B-\f1n gH)|^2=|\n B|^2-\f2n\sum \langle B_{ijk},\de_{ij}H_k\rangle+\f1n|\n H|^2\\
=&|\n B|^2-\f2{3n}\sum \langle B_{ijk},(n+2)E_{ijk}\rangle+\f1n|\n H|^2\\
=&|\n B|^2-\f{2(n+2)}{3n}|E|^2+\f1n|\n H|^2\\
=&|\n B|^2-\f{2}{n}|\n H|^2+\f1n|\n H|^2=|\n B|^2-\f1n|\n H|^2.\\
\endaligned
\end{equation}

\begin{thm}
Let $M^2$ be a complete self-shrinker in
$\R^{3}$, if $|\dot{B}|$ is a constant on $M$, then $|B|^2\equiv0$ or $\f12$.
\end{thm}
\begin{proof}
By \eqref{LB1} and \eqref{LH}, we have
\begin{equation}\aligned\label{LdB}
\mathcal{L}|\dot{B}|^2=\mathcal{L}|B|^2-\f12\mathcal{L}|H|^2=2|\n \dot{B}|^2+2|\dot{B}|^2(\f12-|B|^2).
\endaligned
\end{equation}
Let $\dot{B}_{e_ie_j}=\dot{h}_{ij}\nu$, and the matrix $\dot{h}_{ij}$ can be diagonalized by $\dot{h}_{11}=\la$, $\dot{h}_{22}=-\la$, and $\dot{h}_{12}=0$. Then
$$\dot{h}_{11k}+\dot{h}_{22k}=0\ \ \mathrm{for} \ k=1,2,$$
and
\begin{equation}\aligned\label{0trB}
4|\dot{B}|^2|\n|\dot{B}||^2=&|\n|\dot{B}|^2|^2=4\sum_k(\sum_{i,j}\dot{h}_{ij}\dot{h}_{ijk})^2=8\la^2\sum_k(\dot{h}_{11k}^2+\dot{h}_{22k}^2)\\
=&4\la^2|\n\dot{B}|^2=2|\dot{B}|^2|\n\dot{B}|^2.
\endaligned
\end{equation}
If $|\dot{B}|=0$, then we  complete the proof. Now we suppose that
$|\dot{B}|$ is a positive constant. Then by \eqref{LdB} and
\eqref{0trB}, we get $|B|^2\equiv\f12$.

\end{proof}

Now, we give a result on finite integral properties about derivatives of second fundamental form, which is useful in the later integral estimates.
\begin{pro}
Let $M$ be a complete properly immersed self-shrinker in
$\R^{n+1}$, if $|B|$ is bounded on $M$, then $\int_M|\nabla^2
B|^2\rho<\infty$ and $\int_M|\nabla B|^p\rho<\infty$ for $0\leq
p\leq4$.
\end{pro}
\begin{proof}
Let $\eta$ be an arbitrary smooth function with compact support on
$M$, by (\ref{LB1}) we have
\begin{equation}\aligned\label{4.1}
\int_M|\nabla B|^2\eta^2\rho=&\int_M|B|^2(|B|^2-\frac{1}{2})\eta^2\rho+\frac{1}{2}\int_M(\mathcal{L}|B|^2)\eta^2\rho\\
=&\int_M|B|^2(|B|^2-\frac{1}{2})\eta^2\rho-2\int_M(\nabla|B|\cdot\nabla\eta)|B|\eta\rho\\
\leq&\int_M|B|^2(|B|^2-\frac{1}{2})\eta^2\rho+\epsilon\int_M|\nabla|B||^2\eta^2\rho+\frac{1}{\epsilon}\int_M|B|^2|\nabla\eta|^2\rho.\\
\endaligned
\end{equation}
Since $|B|$ is bounded and $M$ has Euclidean volume growth \cite{DX2}, then
by \eqref{4.1}, we get
\begin{equation}\aligned\label{4.2}
\int_M|\nabla
B|^2\rho\le\int_M|B|^2(|B|^2-\frac{1}{2})\rho<\infty.
\endaligned
\end{equation}
Using this argument for \eqref{2 diff}, we get
\begin{equation}\aligned\label{4.3}
\int_M|\nabla^2 B|^2\rho\le&\int_M(|B|^2-1)|\nabla
B|^2\rho+3\int_M\sum_{i,j,k}h_{ijk}^2(\lambda^2_k-2\lambda_i\lambda_j)\rho+\frac{3}{2}\int_M|\nabla|B|^2|^2\rho\\
<&\infty.
\endaligned
\end{equation}
For any $q\geq0$, multiplying $|\nabla B|^q\eta^2$ on the both
sides of (\ref{LB1}), and integrating by parts, we obtain
\begin{equation}\aligned\label{4.4}
&\int_M|\nabla B|^{2+q}\eta^2\rho=\int_M|B|^2(|B|^2-\frac{1}{2})|\nabla B|^q\eta^2\rho+\frac{1}{2}\int_M(\mathcal{L}|B|^2)|\nabla B|^q\eta^2\rho\\
=&\int_M|B|^2(|B|^2-\frac{1}{2})|\nabla B|^q\eta^2\rho-\frac{1}{2}\int_M\nabla|B|^2\cdot\nabla(|\nabla B|^q\eta^2)\rho\\
\leq&\int_M|B|^2(|B|^2-\frac{1}{2})|\nabla B|^q\eta^2\rho+q\int_M|B|\cdot|\nabla B|^q|\nabla^2B|\eta^2\rho+\int_M|B|\cdot|\nabla B|^{q+1}|\nabla\eta^2|\rho\\
\le&C\int_M|\n B|^q\e^2\r+\f12\int_M|\n
B|^{2q}\e^2\r+C\int_M|\n^2 B|^2\e^2\r
+C\int_M|\nabla B|^{q+1}|\nabla\eta^2|\rho,\\
\endaligned
\end{equation}
where we have used Young's inequality in the last inequality of \eqref{4.4}. By \eqref{4.2} and \eqref{4.3}, we know
$\int_M|\nabla B|^{3}\rho<\infty$ for $q=1$ in \eqref{4.4} and
$\int_M|\nabla B|^{4}\rho<\infty$ for $q=2$ in \eqref{4.4}. By
H$\mathrm{\ddot{o}}$lder inequality, we get this Proposition.
\end{proof}

In what follows, we always denote $S=|B|^2.$ Define
$$f=\sum_{i,j}(\lambda_i-\lambda_j)^2\lambda_i^2\lambda_j^2,\quad f_3=\sum_i\la_i^3,\quad f_4=\sum_i\la_i^4,$$
where $h_{ij}=\la_i\de_{ij}$ at the considered point. Then
$$f=2(Sf_{4}-f_{ 3}^{ 2}).$$
It is a higher order invariant of the second fundamental form.

\begin{lem}\label{intpart}
$$\int_M \Xi\rho=\frac{1}{2}\int_M f\rho-\frac{1}{4}\int_M|\nabla S|^2\rho.$$
\end{lem}
\begin{proof}
By Stokes formula, we have
\begin{equation}\aligned\label{4.5}
-\int_M\sum_{i,j}(h_{ij}\rho)_j(f_3)_i=\int_M\sum_{i,j}h_{ij}(f_3)_{ij}\rho.
\endaligned
\end{equation}
Since $H=-\f{\langle X,\nu\rangle}2$, then
\begin{equation}\aligned\label{4.6}
\sum_{j}(h_{ij}\rho)_j=\sum_jh_{jji}\rho-\sum_jh_{ij}\frac{\langle
X,e_j\rangle}{2}\rho=-\frac{e_i\langle
X,\nu\rangle}{2}\rho+\frac{\langle
X,\nabla_{e_i}\nu\rangle}{2}\rho=0,
\endaligned
\end{equation}
and combining the Ricci identity (\ref{Ric id})
$$t_{ij}=h_{ijij}-h_{jiji}=\la_i\la_j(\la_i-\la_j)\qquad \forall i, j,$$
we have
\begin{equation}\aligned\label{4.7}
&(f_3)_{ll}=\sum_{i,j,k}(3h_{ijl}h_{jk}h_{ki})_l=3\sum_{i,j,k}(h_{ijll}h_{jk}h_{ik}+h_{ijl}h_{kjl}h_{ki}+h_{ijl}h_{jk}h_{ikl})\\
=&3\sum_ih_{iill}\lambda_i^2+6\sum_{i,j}h_{ijl}^2\lambda_i
=3\sum_i\left(h_{llii}\lambda_i^2+\lambda_i^3\lambda_l(\lambda_i-\lambda_l)\right)+6\sum_{i,j}h_{ijl}^2\lambda_i.
\endaligned
    \end{equation}
Combining \eqref{4.5}-\eqref{4.7}, we get
\begin{equation}\aligned\label{4.8}
0&=\int_M\sum_{i,l}\left(h_{llii}\lambda_i^2\lambda_l+\lambda_i^3\lambda_l^2(\lambda_i-\lambda_l)\right)\rho+2\int_M\sum_{i,j,l}h_{ijl}^2\lambda_i\lambda_l\rho\\
&=\frac{1}{2}\int_M\sum_{i}\lambda_i^2S_{ii}\rho-\int_M\sum_{i,j,k}\lambda_i^2h_{jki}^2\rho+\int_M\sum_{i,l}\lambda_i^3\lambda_l^2(\lambda_i-\lambda_l)\rho
+2\int_M\sum_{i,j,l}h_{ijl}^2\lambda_i\lambda_l\rho\\
&=\frac{1}{2}\int_M\sum_{i,j,k}h_{ik}h_{jk}S_{ij}\rho-\int_M \Xi\rho+\frac{1}{2}\int_M f\rho.\\
\endaligned
\end{equation}
By Stokes formula and \eqref{4.6},
\begin{equation}\aligned\label{4.9}
\int_M\sum_{i,j,k}h_{ik}h_{jk}S_{ij}\rho=-\int_M\sum_{i,j,k}(h_{ik}h_{jk}\rho)_jS_{i}=-\int_M\sum_{i,j,k}h_{ijk}h_{jk}S_{i}\rho=-\frac{1}{2}\int_M|\nabla
S|^2\rho.
\endaligned
\end{equation}
Combining \eqref{4.8} and \eqref{4.9}, we complete the proof.
\end{proof}

\begin{lem}\label{ptest}
If notations are as above, then $3\Xi\leq (S+C_1f^{1/3})|\nabla
B|^2$, here $C_1=\frac{2\sqrt{6}+3}{\sqrt[3]{21\sqrt{6}+103/2}}$.
\end{lem}
\begin{proof}
For any three distinct positive integers $i,j,k\in\{1,\cdots,n\}$,
if $\la_i\la_j\le0$ and $\la_i\la_k\le0$, then by Cauchy inequality,
$2|\lambda_i\lambda_j|\leq\frac{1}{2}(\lambda_i-\lambda_j)^2$ which
implies
$|\lambda_i\lambda_j|^3\leq\frac{1}{4}(\lambda_i-\lambda_j)^2\lambda_i^2\lambda_j^2,$
and
$$(|\lambda_i\lambda_j|+|\lambda_i\lambda_k|)^3\leq4(|\lambda_i\lambda_j|^3+|\lambda_i\lambda_k|^3)\leq
(\lambda_i-\lambda_j)^2\lambda_i^2\lambda_j^2+(\lambda_i-\lambda_k)^2\lambda_i^2\lambda_k^2\leq\frac{f}{2}.$$

Since there must be a nonnegative number in three number
$\{\lambda_i\lambda_j,\lambda_j\lambda_k,\lambda_i\lambda_k\}$,
we always have
\begin{equation}\aligned\label{4.10}
-(\lambda_i\lambda_j+\lambda_j\lambda_k+\lambda_i\lambda_k)\leq
(\frac{f}{2})^{1/3}.
\endaligned
\end{equation}

On the other hand, by a simple computation, the function
$\zeta(x)\triangleq x^2(1+x)^2(4x-1)^{-3}$ on
$(\frac{1}{4},+\infty)$ attains its minimum at
$x=1+\sqrt{\frac{3}{2}}$. If $\lambda_j=-x\lambda_i$, then
\begin{equation}\aligned\label{4.11}
(-\lambda_i^2-4\lambda_i\lambda_j)^3&=(4x-1)^3\lambda_i^6\leq\f{x^2(1+x)^2}{\zeta(1+\sqrt{1.5})}\lambda_i^6\\
&=\f1{\zeta(1+\sqrt{1.5})}(\lambda_i-\lambda_j)^2\lambda_i^2\lambda_j^2\leq\frac{f}{2\zeta(1+\sqrt{1.5})}.
\endaligned
\end{equation}
Let $C_1=\frac{2\sqrt{6}+3}{\sqrt[3]{21\sqrt{6}+51.5}}$, by the
definition of $\Xi$ and \eqref{4.10}, \eqref{4.11}, we have
\begin{equation}\aligned\nonumber
3\Xi\leq&\sum_{i,j,k\ distinct}h_{ijk}^2(\lambda_i^2+\lambda_j^2+\lambda_k^2-2(\lambda_i\lambda_j+\lambda_j\lambda_k+\lambda_i\lambda_k))+3\sum_{j,i\neq j}h_{iij}^2(\lambda_j^2-4\lambda_i\lambda_j)\\
\leq&\sum_{i,j,k\ distinct}h_{ijk}^2(S+\sqrt[3]{4}f^{1/3})+3\sum_{j,i\neq j}h_{iij}^2(\lambda_i^2+\lambda_j^2+\sqrt[3]{\frac{f}{2\zeta(1+\sqrt{1.5})}}\ )\\
\leq&(S+C_1f^{1/3})|\nabla B|^2.
\endaligned
\end{equation}
\end{proof}
By the previous definition of
 $t_{ij}$,
\begin{equation}\aligned\label{4.12}
\sum_{i,j,k,l}h_{ijkl}^2\geq&3\sum_{i\neq j}h_{ijij}^2=3\sum_{i< j}(h_{ijij}^2+(h_{ijij}-t_{ij})^2)\\
=&3\sum_{i\neq
j}(h_{ijij}-\frac{t_{ij}}{2})^2+\frac{3}{4}\sum_{i\neq
j}t_{ij}^2\geq\frac{3}{4}\sum_{i\neq
j}(\lambda_i-\lambda_j)^2\lambda_i^2\lambda_j^2.
\endaligned
\end{equation}

Now, we are in a position to prove a second gap property for
self-shrinkers.

\begin{thm}\label{sec}
Suppose that $M^n$ is a complete properly immersed self-shrinker in
$\R^{n+1}$, there exists a positive number $\delta=0.011$ such that
if $\frac{1}{2}\leq|B|^2\leq\frac{1}{2}+\delta$, then
$|B|^2\equiv\frac{1}{2}$.
\end{thm}
\begin{proof}
By \eqref{2 diff}, \eqref{4.12} and Lemma \ref{intpart},
\ref{ptest}, for some fixed $0<\theta<1$ to be defined later, we
have
\begin{equation}\aligned\label{4.13}
&\frac{3}{4}(1-\theta)\int_Mf\rho+\frac{3}{8}\theta\int_M|\nabla S|^2\rho\\
=&\frac{3}{4}\int_Mf\rho-\frac{3}{2}\theta\int_M\Xi\rho
\leq\int_M|\nabla^2B|^2\rho-\frac{3}{2}\theta\int_M\Xi\rho\\
=&\int_M(S-1)|\nabla B|^2\rho+3(1-\frac{\theta}{2})\int_M\Xi\rho+\frac{3}{2}\int_M|\nabla S|^2\rho\\
\leq&\int_M(S-1)|\nabla B|^2\rho+(1-\frac{\theta}{2})\int_M(S+C_1f^{1/3})|\nabla B|^2\rho+\frac{3}{2}\int_M|\nabla S|^2\rho\\
\leq&\int_M((2-\frac{\theta}{2})S-1)|\nabla B|^2\rho+\frac{3}{2}\int_M|\nabla S|^2\rho+\frac{3}{4}(1-\theta)\int_Mf\rho\\
&+\frac{4}{9}C_1^\frac{3}{2}(1-\frac{\theta}{2})^\frac{3}{2}(1-\theta)^{-\frac{1}{2}}\int_M|\nabla B|^3\rho,\\
\endaligned
\end{equation}
where we have used Young's inequality in the last step of the above
inequality, then
\begin{equation}\aligned\label{4.14}
0\leq&\int_M((2-\frac{\theta}{2})S-1)|\nabla
B|^2\rho+(\frac{3}{2}-\frac{3\theta}{8})\int_M|\nabla
S|^2\rho+C_2(n,\theta)\int_M|\nabla B|^3\rho,
\endaligned
\end{equation}
where
$C_2=C_2(n,\theta)=\frac{4}{9}C_1^\frac{3}{2}(1-\frac{\theta}{2})^\frac{3}{2}(1-\theta)^{-\frac{1}{2}}.$

By \eqref{LB1}, for some $\epsilon>0$ to be defined later, we have
\begin{equation}\aligned\label{4.15}
\int_M|\nabla B|^3\rho=&\int_MS(S-\frac{1}{2})|\nabla B|\rho+\frac{1}{2}\int_M(|\nabla B|\mathcal{L}S)\rho\\
=&\int_MS(S-\frac{1}{2})|\nabla B|\rho-\frac{1}{2}\int_M(\nabla|\nabla B|\cdot\nabla S)\rho\\
\leq&\int_MS(S-\frac{1}{2})|\nabla B|\rho+\epsilon\int_M|\nabla^2
B|^2\rho+\frac{1}{16\epsilon}\int_M |\nabla S|^2\rho.
\endaligned
\end{equation}
Combining \eqref{2 diff} and \eqref{pt}, we obtain
\begin{equation}\aligned
\int_M|\nabla^2B|^2\rho\leq\int_M(\frac{\sqrt{17}+3}{2}S-1)|\nabla
B|^2\rho+\frac{3}{2}\int_M|\nabla S|^2\rho,
\endaligned\nonumber
\end{equation}
with the help of the above inequality, \eqref{4.15} becomes
\begin{equation}\aligned\label{4.16}
\int_M|\nabla B|^3\rho\leq&\int_MS(S-\frac{1}{2})|\nabla B|\rho+\epsilon\int_M(\frac{\sqrt{17}+3}{2}S-1)|\nabla B|^2\rho\\
&+(\frac{3\epsilon}{2}+\frac{1}{16\epsilon})\int_M |\nabla S|^2\rho.
\endaligned
\end{equation}
Multiplying $S$ on the both sides of (\ref{LB1}), and integrating
by parts, we see
\begin{equation}\aligned\label{4.17}
\frac{1}{2}\int_M|\nabla S|^2\rho=&\int_MS^2(S-\frac{1}{2})\rho-\int_MS|\nabla B|^2\rho\\
=&\int_MS(S-\frac{1}{2})^2\rho+\frac{1}{2}\int_MS(S-\frac{1}{2})\rho-\int_MS|\nabla B|^2\rho\\
=&\int_M(\frac{1}{2}-S)|\nabla
B|^2\rho+\int_MS(S-\frac{1}{2})^2\rho.
\endaligned
\end{equation}
Combining \eqref{4.14}, \eqref{4.16} and \eqref{4.17}, we get
\begin{equation}\aligned\label{4.18}
0\leq&\int_M\left((2-\frac{\theta}{2})S-1+C_2\epsilon(\frac{\sqrt{17}+3}{2}S-1)\right)|\nabla B|^2\rho
+C_2\int_MS(S-\frac{1}{2})|\nabla B|\rho\\
&+\left(\frac{3}{2}-\frac{3\theta}{8}+C_2(\frac{3\epsilon}{2}+\frac{1}{16\epsilon})\right)\int_M|\nabla S|^2\rho\\
=&\int_M\left((2-\frac{\theta}{2})S-1+C_2\epsilon(\frac{\sqrt{17}+3}{2}S-1)\right)|\nabla B|^2\rho+C_2\int_MS(S-\frac{1}{2})|\nabla B|\rho\\
&+\left(3-\frac{3\theta}{4}+C_2(3\epsilon+\frac{1}{8\epsilon})\right)\left(\int_M(\frac{1}{2}-S)|\nabla B|^2\rho+\int_MS(S-\frac{1}{2})^2\rho\right)\\
=&\int_M\left(-\frac{\theta}{4}+\frac{\sqrt{17}-1}{4}C_2\epsilon-(1-\frac{\theta}{4}-\frac{\sqrt{17}-3}{2}C_2\epsilon+\frac{C_2}{8\epsilon})
(S-\frac{1}{2})\right)|\nabla B|^2\rho\\
&+C_2\int_MS(S-\frac{1}{2})|\nabla B|\rho+\left(3-\frac{3\theta}{4}+C_2(3\epsilon+\frac{1}{8\epsilon})\right)\int_MS(S-\frac{1}{2})^2\rho.\\
\endaligned
\end{equation}
By Cauchy-Schwartz inequality and \eqref{LB1}, we have
\begin{equation}\aligned\label{4.19}
&\int_MS(S-\frac{1}{2})|\nabla B|\rho\\
\leq&2(\frac{1}{2}+\delta)\epsilon\int_MS(S-\frac{1}{2})\rho+\frac{1}{8(1/2+\delta)\epsilon}\int_MS(S-\frac{1}{2})|\nabla B|^2\rho\\
=&\int_M\left((1+2\delta)\epsilon+\frac{S(S-1/2)}{8(1/2+\delta)\epsilon}\right)|\nabla
B|^2\rho\\
\leq&\int_M\left((1+2\delta)\epsilon+\frac{S-1/2}{8\epsilon}\right)|\nabla
B|^2\rho.
\endaligned
\end{equation}
Combining (\ref{LB1}), \eqref{4.18} and \eqref{4.19}, we have
\begin{equation}\aligned
0\leq&\int_M\left(-\frac{\theta}{4}+\frac{\sqrt{17}-1}{4}C_2\epsilon-(1-\frac{\theta}{4}-\frac{\sqrt{17}-3}{2}C_2\epsilon+\frac{C_2}{8\epsilon})
(S-\frac{1}{2})\right)|\nabla B|^2\rho\\
&+C_2\int_M\left((1+2\delta)\epsilon+\frac{S-1/2}{8\epsilon}\right)|\nabla B|^2\rho\\
&+\left(3-\frac{3\theta}{4}+C_2(3\epsilon+\frac{1}{8\epsilon})\right)\delta\int_MS(S-\frac{1}{2})\rho\\
=&\int_M\left(-\frac{\theta}{4}+C_2\epsilon(\frac{\sqrt{17}+3}{4}+5\delta)+(3-\frac{3}{4}\theta)\delta+\frac{C_2\delta}{8\epsilon}\right)|\nabla B|^2\rho\\
&-\int_M(1-\frac{\theta}{4}-\frac{\sqrt{17}-3}{2}C_2\epsilon)
(S-\frac{1}{2})|\nabla B|^2\rho.\\
\endaligned
\end{equation}
Let $\epsilon=\sqrt{\frac{\delta}{2(\sqrt{17}+3)+40\delta}}$,
$\theta=1/2$, then $C_2=\f{\sqrt{6}}{6}C_1^{\f{3}{2}}\leq0.8933$,
and $\frac{7}{8}-\frac{\sqrt{17}-3}{2}C_2\epsilon>0$,
\begin{equation}\aligned
0\leq(-\frac{1}{8}+\frac{0.8933}{\sqrt{2}}\sqrt{\delta(\frac{\sqrt{17}+3}{4}+5\delta)}+\frac{21}{8}\delta)\int_M|\nabla B|^2\rho.\\
\endaligned
\end{equation}
If we choose $\delta=0.011$, then $|\nabla B|\equiv0$.
\end{proof}

\bigskip

\bibliographystyle{amsplain}

\end{document}